\documentclass[11pt]{article}

\usepackage[a4paper,margin=0.6in]{geometry}
\usepackage{amsmath,amssymb,amsthm,amsfonts,mathtools,amscd}
\usepackage{mathrsfs}
\usepackage{bm}
\usepackage{enumitem}
\usepackage{hyperref}
\usepackage{color}
\usepackage{cite}
\usepackage{esint}

\hypersetup{
  colorlinks=true,
  linkcolor=blue,
  citecolor=blue,
  urlcolor=blue
}

\usepackage{setspace}
\usepackage[toc]{appendix}

\numberwithin{equation}{section}

\newtheorem{theorem}{Theorem}[section]
\newtheorem{lemma}[theorem]{Lemma}
\newtheorem{proposition}[theorem]{Proposition}

\theoremstyle{definition}
\newtheorem{definition}[theorem]{Definition}
\newtheorem{remark}[theorem]{Remark}

\newcommand{\T}{\mathbb{T}}
\newcommand{\E}{\mathbb{E}}
\newcommand{\Prob}{\mathbb{P}}
\newcommand{\dd}{\,\mathrm{d}}
\newcommand{\ep}{\varepsilon}
\newcommand{\La}{\Lambda}

\newcommand{\mc}{\mathcal}

\newcommand{\br}[1]{\left(#1\right)}

\newcommand{\norm}[1]{\left\|#1\right\|}
\newcommand{\ip}[2]{\left\langle #1,#2\right\rangle}
\newcommand{\dx}{\,\mathrm{d}x}

\theoremstyle{plain}
\newtheorem{THM}{Theorem}

\title{Mixed Third-Order Flux Laws for Dual Cascade
in the Stochastic SQG Equation}
\author{
Rongchang Liu,\;
Kening Lu,\;
Zexin Wang
\\[2mm]
{\small School of Mathematics, Sichuan University,
Chengdu, Sichuan 610064, PR China}
\\
{\small Emails: rcliu@scu.edu.cn,\;
keninglu@scu.edu.cn,\;
wangzexin@stu.scu.edu.cn (corresponding author)}
}

\date{}

\usepackage{xcolor}

\hypersetup{
    colorlinks=true,
    linkcolor=blue,
    citecolor=blue,
    urlcolor=blue,
    pdfstartview=FitH,
    bookmarksopen=true,
    bookmarksnumbered=true
}

\begin{document}

\maketitle
\begingroup

\endgroup

\begingroup
\renewcommand{\thefootnote}{}
\footnotetext{
This work was supported by the National Natural Science Foundation of China
(Grant Nos.~12090010 and 12090013) and the Fundamental Research Funds for
the Central Universities.
}
\endgroup

\begin{abstract}
We study dual-cascade flux laws for the stochastic forced--dissipative
surface quasi-geostrophic (SQG) equation on a large periodic box. For
statistically stationary solutions, under a weak anomalous dissipation
assumption, we derive rigorous mixed third-order structure-function laws for
the dual cascade: a Yaglom-type law for the direct cascade of surface
potential energy (SPE) and a mixed flux law for the inverse
cascade of the Hamiltonian. 
We further establish complementary scaling laws: a cubic small-scale law for the Hamiltonian structure function and a large-scale vanishing law for the SPE structure function.
In particular, both Hamiltonian laws appear to be new for SQG even as an explicit
third-order structure-function relation.
We also prove Onsager-type obstruction results
showing that sufficiently regular stationary families cannot sustain the
corresponding non-zero fluxes: $B^s_{3,\infty}$-regularity above the
Onsager threshold $1/3$ rules out the direct SPE flux, while sufficient
low-frequency Besov regularity rules out the inverse Hamiltonian flux. These
results provide a rigorous formulation of the SQG dual-cascade phenomenology
in a stochastic stationary setting.

\end{abstract}



\

\section{\textbf{Introduction}}  
\setcounter{equation}{0}

The surface quasi-geostrophic (SQG) equation, whose roots can be traced back
to the work of Blumen \cite{ref-SQG-turbulence-1978},
is a classical active-scalar model in geophysical fluid dynamics 
that describes the transport and evolution of surface potential temperature in 
rapidly rotating, stably stratified fluids. 
It has been widely used in the study of atmospheric and oceanic fronts, 
turbulent cascades, and possible finite-time singularity mechanisms. 
Although SQG is two-dimensional, its nonlocal active-scalar structure exhibits
an amplification mechanism reminiscent of vortex stretching in the 3D Euler
equations, making it an important model problem in mathematical fluid dynamics.
Accordingly, SQG has been extensively studied in works by Constantin, Majda
and Tabak \cite{ref-Majda94}, Córdoba
\cite{ref-Annals-1998}, Constantin and Wu \cite{ref-1999-SIAM}, Kiselev,
Nazarov and Volberg \cite{ref-2007-Invent}, Caffarelli and Vasseur
\cite{ref-2010-Annals}, among others.

In the statistical theory of turbulence, Kolmogorov's theory of
three-dimensional turbulence \cite{ref-K41-1,ref-K41-2,ref-K41-3} and the dual-cascade theory of Kraichnan--Leith--Batchelor
\cite{ref-1967, ref-1968, ref-1969}
for two-dimensional turbulence laid the foundation for the study of turbulent cascades and stimulated extensive subsequent developments. SQG turbulence is
also expected to exhibit a dual-cascade phenomenology, associated with its two quadratic inviscid invariants: a direct cascade of surface potential energy (SPE) toward small scales and an inverse transfer of the Hamiltonian toward large scales. While this picture has been developed in phenomenological and numerical
studies \cite{ref-SQG-turbulence-1994,ref-SQG-turbulence-1995,ref-SQG-turbulence-2017,ref-PRF-2024,ref-2025-JFM,ref-2025-JFM-Yaglom}, a rigorous derivation of the corresponding third-order flux laws for SQG has remained largely open. In particular, although spectral predictions for SQG
turbulence are available, to the best of our knowledge no explicit
third-order flux law for the inverse Hamiltonian transfer has been established. 

Motivated by this dual-cascade phenomenology, we study statistically stationary SQG turbulence in a stochastic forced--dissipative setting. Under a weak
anomalous dissipation assumption, we derive rigorous mixed third-order flux
laws for both the direct SPE cascade and the inverse Hamiltonian cascade. We also prove Onsager-type 
obstruction results showing that sufficiently 
regular stationary families cannot sustain 
the corresponding non-zero fluxes. 

More precisely, consider the stochastic SQG equation 
with small-scale viscous dissipation,
stochastic forcing, and large-scale damping
on the periodic box $\mathbb{T}_\lambda^2:=
\mathbb{R}^2/(\lambda\mathbb{Z})^2$ with size $\lambda>0$:
\begin{equation}\label{eq:(1.1)}
d\theta +( u \cdot \nabla \theta
-\nu\Delta\theta+\alpha\Lambda ^{-2\gamma}
\theta)\,dt
= \sum_{n\in\mathbb{N}} g_n^\lambda\, dW_n(t), \qquad u = \mathcal{R}^{\perp}\theta=(-\mathcal{R}_2\theta,\mathcal{R}_1\theta). 
\end{equation}
Here $\Lambda:=(-\Delta)^{\frac{1}{2}}, \gamma\geq 0$ and $\nu,\alpha> 0$.  The scalar $\theta$ denotes 
the surface potential temperature,
$u$ is the associated 
incompressible velocity field,
and $\mathcal{R}_j:=\partial_j(-\Delta)^{-\frac{1}{2}},\ j = 1, 2$ denotes the $j$-th Riesz transform.
The processes
$\{W_n\}_{n\in\mathbb{N}}$ are i.i.d. standard Brownian
motions over a filtered probability space $(\Omega,\mathcal{F},\left(\mathcal{F}_t\right),\mathbb{P})$ satisfying the usual conditions, and
$\{g^\lambda_n\}_{n\in\mathbb{N}}$
is a sequence of smooth
scalar fields with zero mean. For the inviscid SQG equation, the two quadratic invariants relevant to the
dual-cascade picture are the surface potential energy
$$
E(\theta)=\|\theta\|_{L^2}^2
$$
and the Hamiltonian
$$
H(\theta)=\langle\theta,\psi\rangle
=\|\Lambda^{-1/2}\theta\|_{L^2}^2,
$$
where $\psi=\Lambda^{-1}\theta$ denotes the stream function.
The former is one order higher than the latter, leading phenomenologically to
a direct SPE cascade and an inverse Hamiltonian transfer.

Our first main result establishes mixed third-order flux laws for the SQG dual-cascade under a weak anomalous dissipation assumption. An informal version of the result is summarized as follows. Here we denote by $\delta_h f(x) = f(x+h)-f(x)$ the increment of $f$ by $h\in\mathbb R^2$. 
\begin{THM}\label{THM-A}
Let $\{\theta^{\nu,\alpha}\}_{\nu,\alpha>0}$ be a family of statistically
stationary solutions to \eqref{eq:(1.1)}. Under a weak anomalous dissipation assumption, assuming $\nu\sim\alpha$, and taking the large-box inviscid limit
$\nu,\alpha\to0$ with $\lambda=\lambda(\nu,\alpha)\to\infty$, there exist a
dissipative scale $\ell_{\nu}\to0$ and a damping scale $\ell_\alpha\to\infty$
such that
$$S_E(\ell)
\to -2\varepsilon\ell,\ \ \ell_\nu\ll \ell\ll 1,$$
and 
$$S_H(\ell)
\to\eta\ell,\ \ 1\ll \ell\ll\ell_\alpha,$$
where $\varepsilon$ and $\eta$ denote the averaged input rates of SPE and
Hamiltonian, respectively, and the quantity 
$$
S_E(\ell) := \mathbb{E} 
\fint_{\mathbb{S}^1}\fint_
{\mathbb{T}_\lambda^2}  
|\delta_{\ell \hat{n}} \theta|^2 
(\delta_{\ell \hat{n}} u \cdot \hat{n} ) \, dx 
\, dS(\hat{n})$$
is the Yaglom-type mixed 
structure function associated with the direct SPE cascade, while 
$$S_H(\ell) := \mathbb{E} 
\fint_{\mathbb{S}^1}
\fint_
{\mathbb{T}_\lambda^2} \delta_{\ell\hat{n}}\theta
 \delta_{\ell\hat{n}}\psi
 (\delta_{\ell\hat{n}}u \cdot \hat{n}) \, dx \, dS(\hat{n})$$
is the mixed structure function associated 
with the inverse Hamiltonian cascade. 
\end{THM}
We refer the reader to Theorem \ref{Thm} for the precise statement. Theorem \ref{THM-A} turns the SQG cascade picture into precise flux laws.
For the direct SPE cascade, Yaglom-type mixed third-order relations have appeared
in phenomenological discussions and numerical studies
\cite{ref-PRF-2024, ref-2025-JFM-Yaglom},
but they have not been established in a mathematically rigorous framework.
More notably, for the inverse Hamiltonian
cascade, there has been no explicit SQG third-order structure-function
law even at the phenomenological or numerical level. 
Our work fills this gap by rigorously establishing the direct 
Yaglom-type law and deriving
a precise mixed third-order flux law for the 
inverse Hamiltonian transfer.

These dual flux laws are formally consistent with the KLB-type spectral
predictions proposed in the early works of Blumen
\cite{ref-SQG-turbulence-1978}, Pierrehumbert, Held and Swanson
\cite{ref-SQG-turbulence-1994}, and Held et al.
\cite{ref-SQG-turbulence-1995}. In terms of the isotropic shell spectra, these predictions read: 
$$|k|\mathbb E|\widehat\theta(k)|^2
\sim\varepsilon^{2/3}|k\vert^{-5/3},
\quad 1\ll|k\vert\ll\ell_\nu^{-1},$$
$$|k|\mathbb E\bigl(|k|^{-1}|\widehat\theta(k)|^2\bigr)\sim\eta^{2/3}|k\vert^{-2},
\quad \ell_\alpha^{-1}\ll|k\vert\ll1,$$
where $k$ is the frequency and $\widehat\theta$ is the Fourier transform of $\theta$.   
These spectral predictions should be distinguished from Tran's non-cascading results \cite{Tran}, where fixed critical dissipation bounds SPE and obstructs persistent inverse transfer. This does not contradict our flux laws, derived under W.A.D. as viscosity and large-scale damping vanish, without a uniform SPE bound.

Beyond these non-zero flux laws, we also obtain complementary asymptotics for the two structure functions in the opposite inertial ranges.

\begin{THM}\label{THM-B}
    Under the assumptions of Theorem~\ref{THM-A}, there exists a dissipative scale $\widetilde{\ell}_\nu\to0$
    such that
$$
S_H(\ell)\to\frac{\chi}{8}\ell^3,
\quad
\widetilde{\ell}_\nu\ll\ell\ll1,$$
and
$$
S_E(\ell)\to0,
\quad
1\ll\ell\ll\ell_\alpha,$$
where $\chi$ denotes the weighted forcing intensity,
defined in Section~\ref{SECTION2.3}.
\end{THM}
We refer the reader to Theorem~\ref{other laws} for the precise statement.
These laws are consistent with the dual-cascade picture: in the direct-cascade range, the Hamiltonian structure function obeys a cubic law and is thus of higher order and smaller than the linear SPE law as $\ell\to0$, while in the inverse-cascade range the SPE structure function vanishes.

In a complementary direction, we also prove Onsager-type obstruction results,
which provide necessary regularity conditions for non-trivial third-order flux
laws. Informally, a non-zero direct SPE flux requires roughness at the Onsager
scale $1/3$, while a non-zero inverse Hamiltonian flux requires critical
large-scale, low-frequency accumulation. 

\begin{THM}\label{THM-C}
Let $\{\theta^{\nu,\alpha}\}_{\nu,\alpha>0}$ be a family of statistically
stationary solutions to \eqref{eq:(1.1)}, considered in the large-box inviscid
limit $\nu,\alpha\to0$ with $\lambda=\lambda(\nu,\alpha)\to\infty$. The
following obstruction results hold. 

\begin{enumerate}[label=\emph{(\roman*)}]
\item \emph{(Direct SPE flux)} If the family is uniformly bounded in
$L^3(\Omega;B^s_{3,\infty})$ for some $s>1/3$, then the direct SPE flux
vanishes:
$$
\frac{S_E(\ell)}{\ell}\to0
$$
uniformly in the direct-cascade range $\ell_\nu\ll\ell\ll1$. Consequently,
no non-trivial direct SPE flux law can hold along such a uniformly regular
family.

\item \emph{(Inverse Hamiltonian flux)} If the low-frequency part of the family
is uniformly better than the critical $\dot B^0_{3,\infty}$ behavior, then the
inverse Hamiltonian flux vanishes:
$$
\frac{S_H(\ell)}{\ell}\to0
$$
uniformly in the inverse-cascade range $1\ll\ell\ll\ell_\alpha\ll\lambda$. Consequently,
no non-trivial inverse Hamiltonian flux law can hold along such a uniformly
low-frequency regular family.
\end{enumerate}
\end{THM}

We refer the reader to Theorem \ref{Onsager I} and Theorem \ref{Onsager II}
for the precise statements. These obstruction results are naturally connected
with the Onsager-type conservation theory for the inviscid SQG equation. For
the SPE, the Onsager threshold is $1/3$: regularity above this threshold
implies conservation
\cite{ref-Onsager-Zhou2005,ref-Onsager-Chae2006,ref-Onsager-AkramovWiedemann2019},
with endpoint refinements in \cite{ref-Onsager-WangYeYu2022}; in the present
setting, such regularity also suppresses any non-trivial direct SPE flux. For
the Hamiltonian, the sharp Onsager threshold is instead $0$
\cite{ref-Onsager-BuckmasterShkollerVicol2019,ref-Onsager-IsettMa2021,ref-Onsager-DaiGiriRadu2024},
and our inverse obstruction shows that any non-trivial inverse Hamiltonian flux
requires critical large-scale low-frequency accumulation. Recent deterministic works of De Rosa–Latocca–Park \cite{de2025global} and De Rosa–Yuzbasioglu \cite{ref-DeRosaYuzbasioglu26}
study Hamiltonian conservation and the absence of anomalous
Hamiltonian dissipation for vanishing viscosity limits of SQG and gSQG, showing that suitable
compactness in the Hamiltonian topology rules out viscous Hamiltonian anomaly below the Onsager-type
threshold. This is complementary to our statistically stationary setting, where the Hamiltonian part of
weak anomalous dissipation assumes the absence of such viscous anomaly and the inverse flux law identifies the corresponding
large-scale transfer.

Stochastic SQG equations closely related to \eqref{eq:(1.1)} have also been studied
extensively from the viewpoint of stochastic PDE and Markov dynamics \cite{ref-Zhu-2015-AOP, ref-Zhu-2013, ref-2019, ref-2021}. In
particular, R\"ockner, Zhu and Zhu \cite{ref-Zhu-2015-AOP} established
well-posedness and ergodic properties for stochastic dissipative SQG equations
in the subcritical regime. The corresponding invariant measures provide
canonical statistically stationary states for the forced--dissipative SQG
dynamics. From this perspective, Theorems \ref{THM-A}, \ref{THM-B} and \ref{THM-C}
describe structural information on such invariant measures in the large-box
inviscid limit, under the weak anomalous dissipation assumption.

Our results are also in line with a recent rigorous program for turbulent
scaling laws. In passive scalar turbulence \cite{bedrossian2022lagrangian}, mixing and Lagrangian chaos have
been used to verify anomalous dissipation and to derive scaling laws directly
from the underlying stochastic dynamics. For the Navier--Stokes equations,
Bedrossian, Coti Zelati, Punshon-Smith and Weber established sufficient
conditions for the Kolmogorov (4/5)-law in three dimensions
\cite{ref-3D-NS-2019} and for the dual-cascade flux laws in stochastic
two-dimensional Navier--Stokes equations \cite{ref-2D-NS-dual}. The present
work contributes to this direction from the SQG side: under W.A.D., it gives
a rigorous formulation of the SQG dual-cascade flux laws, identifies the
relevant mixed third-order quantities, and clarifies the regularity mechanisms
compatible with non-zero SPE and Hamiltonian fluxes. In this sense, it
provides a rigorous target for future attempts to derive SQG turbulent flux
laws directly from the dynamics.

The rest of the paper is organized as follows. 
In Section \ref{2}, we introduce the notation, set up the forced--dissipative
SQG framework, derive the stationary balance laws, 
state the W.A.D. assumption, and present the main results. 
In Section \ref{3}, we derive the SPE Kármán--Howarth--Monin
(KHM) relation and prove the direct SPE flux law and its large-scale vanishing law. 
In Section \ref{4}, we establish the Hamiltonian 
KHM relation and prove the inverse Hamiltonian flux law and its small-scale cubic law. 
In Section \ref{5}, we prove the Onsager-type obstruction results.

\section{Settings and main results} 
\setcounter{equation}{0}
\label{2}
In this section, we first introduce the necessary notation and technical settings, and then state our main results in detail.

\subsection{Notations and conventions}\label{notations}

\ \ \ \ Throughout the paper,
the scalar fields $\theta(t,x)$ and 
$\{g_n^\lambda(x)\}_{n\in\mathbb N}$
are both assumed 
to have zero spatial mean on 
$\mathbb{T}_\lambda^2$:
$$\int_{\mathbb{T}_\lambda^2}
\theta \,dx=0,\ 
\int_{\mathbb{T}_\lambda^2}
g_n^\lambda \,dx=0.$$

We use the notation $a\lesssim b$ 
to represent $a\leq Cb$ for some constant
$C$,
and $a\gtrsim b$ is defined analogously. 
We write $a\sim b$ if $a$ and
 $b$ are of the same order, i.e.,
$C_1a\leq b\leq C_2a$ for some 
constants $C_1,C_2>0$.

The spatial domain considered in 
this paper is 
$\mathbb{T}^2_\lambda=\mathbb{R}^2/(\lambda\mathbb{Z})^2$, and 
we use $\|f\Vert_{L_\lambda^2}:=
\big(\fint_{\mathbb{T}_\lambda^2}| f \vert^2\,dx\big)^{\frac{1}{2}}$
 to represent the average $L^2$-norm
of $f$ on the torus $\mathbb{T}_\lambda^2$, where 
$\fint_A f \,dx
:=\frac{1}{|A\vert}\int_{A}
f \,dx.$ 
We also use $dS(\hat n)$ to
denote the 
arclength measure on the 
unit circle $\mathbb S^1$.
For $h\in\mathbb{R}^2$,
we use the increment notation
$\delta_hf(x):=f(x+h)-f(x).$

On the torus $\mathbb T^2_\lambda=\mathbb R^2/(\lambda\mathbb Z)^2$,
the frequency lattice is
$
\mathbb Z^2_\lambda:=\frac{2\pi}{\lambda}\mathbb Z^2 .
$
For a periodic function $f$, we use the Fourier convention
$$
f(x)=\sum_{k\in\mathbb Z^2_\lambda}\widehat 
f(k)e^{ik\cdot x},
\qquad
\widehat f(k)
=
\fint_{\mathbb T^2_\lambda}
f(x)e^{-ik\cdot x}\,dx ,
$$
and for $\delta>0$, we use $f_{\le\delta}$ to denote
the 
low-frequency Fourier cutoff
$$
f_{\le\delta}(x)
:=
\sum_{\substack{k\in\mathbb Z^2_\lambda:|k|\le\delta}}
\widehat f(k)e^{ik\cdot x}.
$$

Fourier multipliers are defined with respect to this frequency lattice. In particular,
$\Lambda^s=(-\Delta)^{s/2}$ is given by
$
\widehat{\Lambda^s f}(k)=|k|^s\widehat f(k),
$
and the Riesz transforms are defined by
$
\widehat{R_j f}(k)
=
i\frac{k_j}{|k|}\widehat f(k),\ 
k\neq0 .
$
All negative-order multipliers are understood on 
zero-mean functions, so that the zero
mode is absent.

We will use Littlewood--Paley decomposition on $\mathbb T^2_\lambda$.  Let $\vartheta$ and $\varphi$ be standard smooth cutoffs such
that $\vartheta$ is supported near the origin, $\varphi$ is supported in an annulus,
and
$$
\vartheta(r)+\sum_{j\ge 0}\varphi(2^{-j}r)=1,\qquad r\ge 0.
$$
We denote by $\Delta_j f$, $j\ge0$,  the non-homogeneous dyadic blocks, and by
$\Delta_{-1}f$ the low-frequency block. More explicitly,
$$
f=\Delta_{-1}f+\sum_{j\ge0}\Delta_j f,
$$
where
$$
\Delta_j f(x)
=
\sum_{k\in\mathbb Z_\lambda^2}
\varphi(2^{-j}|k|)\widehat f(k)e^{ik\cdot x},
\qquad j\ge0,
$$
and
$$
\Delta_{-1}f(x)
=
\sum_{k\in\mathbb Z_\lambda^2}
\vartheta(|k|)\widehat f(k)e^{ik\cdot x}.
$$
For $s\in\mathbb R$ and $1\le p,q<\infty$, the non-homogeneous Besov norm is
defined by
$$
\|f\|_{B^s_{p,q}}
=
\|\Delta_{-1}f\|_{L^p_\lambda}
+
\left(
\sum_{j\ge0}
\bigl(2^{js}\|\Delta_j f\|_{L^p_\lambda}\bigr)^q
\right)^{1/q}.
$$
When $q=\infty$, this is understood as
$$
\|f\|_{B^s_{p,\infty}}
=
\|\Delta_{-1}f\|_{L^p_\lambda}
+
\sup_{j\ge0}
2^{js}\|\Delta_j f\|_{L^p_\lambda}.
$$

We also use the homogeneous Littlewood--Paley decomposition on the zero-mean subspace. 
Choose the annular cutoff $\varphi$ so that
$\sum_{j\in\mathbb Z}\varphi(2^{-j}r)=1,\ r>0.$
Since all functions under consideration have zero spatial mean, the zero Fourier mode is absent. Hence, for every zero-mean periodic distribution $f$, we have the homogeneous decomposition
$f=\sum_{j\in\mathbb Z}\dot\Delta_j f$
in the sense of distributions, where the homogeneous dyadic blocks are defined by
$$
\dot\Delta_j f(x)
=
\sum_{k\in\mathbb Z_\lambda^2\setminus\{0\}}
\varphi(2^{-j}|k|)\widehat f(k)e^{ik\cdot x}.
$$

For $s\in\mathbb R$ and $1\le p,q<\infty$, the homogeneous Besov seminorm is
defined by
$$
\|f\|_{\dot B^s_{p,q}}
=
\left(
\sum_{j\in\mathbb Z}
\bigl(2^{js}\|\dot\Delta_j f\|_{L^p_\lambda}\bigr)^q
\right)^{1/q}.
$$
When $q=\infty$, this becomes
$$
\|f\|_{\dot B^s_{p,\infty}}
=
\sup_{j\in\mathbb Z}
2^{js}\|\dot\Delta_j f\|_{L^p_\lambda}.
$$

For positive regularity, the case $q=\infty$ also admits an equivalent
difference characterization. In particular, for $0<s<1$,
$$
\|f\|_{B^s_{p,\infty}}
\sim
\|f\|_{L^p_\lambda}
+
\sup_{0<|h|\le1}
\frac{\|\delta_h f\|_{L^p_\lambda}}{|h|^s},
\quad
\|f\|_{\dot{B}_{p,\infty}^s}\sim \sup_{h \neq 0} \frac{
    \|\delta_h f\|_{L^p_\lambda}}{|h|^s}.
$$
Thus $B^s_{p,\infty}$ regularity controls 
increments of order $|h|^s$:
$\|\delta_h f\|_{L^p_\lambda}\lesssim
\|f\|_{B^s_{p,\infty}} |h|^s,\  0<|h|\le1.$

We will also use the following Besov seminorm to characterize the Onsager-type obstruction for the inverse Hamiltonian cascade. 
\begin{definition}
For $s>0$, we define the
low-frequency Besov seminorm
$$
\|f\|^{\mathrm{low}}_{\dot B^{-s}_{3,\infty}(\mathbb T^2_\lambda)}
:=
\sup_{j\leq 0}
2^{-s j}
\|\dot\Delta_j f\|_{L^3_\lambda}.$$
\end{definition}
Equivalently, the condition $\|f\|^{\mathrm{low}}_{\dot B^{-s}_{3,\infty}(\mathbb T^2_\lambda)}<\infty$ says that the low-frequency blocks satisfy
$$
\|\dot\Delta_j f\|_{L^3_\lambda}
\lesssim 2^{s j},
\qquad j\leq 0.
$$
In terms of the large spatial scale $L\sim 2^{-j}$, this corresponds to the decay
$$
\|\dot\Delta_L f\|_{L^3_\lambda}
\lesssim L^{-s},
\qquad L\gg 1,
$$
where $\dot\Delta_L$ denotes the homogeneous dyadic block localized at frequencies $|k|\sim L^{-1}$. 
Thus $\|\cdot\|^{\mathrm{low}}_{\dot B^{-s}_{3,\infty}}$ measures 
low-frequency regularity
in the reciprocal scale $L^{-1}$.

\subsection{ The forced--dissipative SQG setting 
and stationary solutions}\label{2.2 Stationary}
\ \ \ \ In this subsection, we specify the 
forced--dissipative framework used throughout the paper,
together with the assumptions on the forcing 
and the existence of stationary solutions. The forced--dissipative SQG model \eqref{eq:(1.1)} 
provides a 
natural setting for studying statistically 
stationary turbulent states: the stochastic 
forcing continuously injects fluctuations 
into the system, the small-scale viscosity 
$-\nu\Delta\theta$ dissipates high-frequency 
structures, while 
the large-scale Ekman-type damping
$\alpha\Lambda^{-2\gamma}\theta$
acts as a low-frequency sink for the Hamiltonian
that is transferred 
toward large scales, thereby preventing its 
accumulation near the lowest modes.
Such a forced--dissipative SQG model with stochastic 
forcing has also been studied in some physics 
literature;
see \cite{ref-2025-JFM, ref-PRF-2024, ref-2025-JFM-Yaglom, ref-2024-PhD}.

We now assume the following uniform 
regularity and low-frequency conditions on the
forcing fields $\{g_n^\lambda(x)\}_{n\in\mathbb N}$:
\begin{equation}
    \sup_{\lambda\ge1}\sum_n
\|\Lambda^{1+\sigma}g^\lambda_n\|_{L^2_\lambda}^2<\infty,\ 
\forall\sigma>0,
\label{eq:(regularity for noise)}
\end{equation}
\begin{equation}
    \lim_{\delta\to0}\sup_{\lambda\ge1}\sum_n
\|(\Lambda^{-1/2}g^\lambda_n)
_{\leq\delta}\|_{L^2_\lambda}^2=0.
\label{eq:(Low F)}
\end{equation}
The first condition controls the high-frequency tail of the 
forcing uniformly in the box size, 
while the second one prevents the Hamiltonian input from 
concentrating near the lowest modes. 
Thus the forcing acts essentially at finite, $\mathcal{O}(1)$--scales 
rather than 
directly at the box scale.

With the forcing assumptions above, 
we now specify the stationary solutions
used throughout the paper. 
We work with solutions which are strong in 
both senses: 
probabilistically strong, namely
adapted to a prescribed stochastic basis carrying the driving 
Brownian motions, and
strong in the PDE sense, namely possessing the Sobolev regularity
which is 
needed to justify the 
identities and estimates used below.

\begin{definition}
Let $\nu,\alpha>0$ and $\lambda\ge1$ be fixed. Given a stochastic basis 
$\mathcal{S} =\{\Omega, \mathcal{F}, 
(\mathcal{F}_t), \mathbb{P}, \{W_n\}_{n\in\mathbb{N}}\}$, a process $\theta$ is called a statistically
stationary strong solution to  \eqref{eq:(1.1)} if the following conditions hold:

\begin{enumerate}[label=(\roman*)]
\item $\theta$ is $\mathcal{F}_t$-adapted
 and satisfies the stochastic SQG equation \eqref{eq:(1.1)} in a 
pathwise sense with respect to the given stochastic basis $\mathcal{S}$.

\item $\theta$ has the following 
Sobolev regularity:
$$
\theta\in C([0,\infty);H^1(\T_\lambda^2))
\cap L^2_{\rm loc}([0,\infty);H^2(\T_\lambda^2))
\quad \Prob\text{-a.s.}
$$

\item The law of $\theta$ is invariant under time translations: for every $\tau\ge0$,
\begin{equation*}
    \mc L(\theta(\cdot+\tau))=\mc L(\theta(\cdot))
\end{equation*}
as probability measures on the path space.
\end{enumerate}
\end{definition}

The following existence result follows from the standard well-posedness theory
for stochastic SQG equations and the Krylov--Bogoliubov construction of
invariant measures; see, for instance, \cite{ref-Zhu-2015-AOP}. 

\begin{theorem}
\label{Thm2.2}
Consider the aforementioned zero-mean forcing functions
$\{g_n^\lambda\}_{n\in\mathbb N}$ 
satisfying \eqref{eq:(regularity for noise)}. Then, for every fixed
$\nu,\alpha>0$ and $\lambda\ge1$, equation  \eqref{eq:(1.1)} admits a statistically stationary
strong solution $\theta$ globally such that
$$
\theta\in C([0,\infty);H^1(\T_\lambda^2))
\cap L^2_{\rm loc}([0,\infty);
H^2(\T_\lambda^2)),
\quad \Prob\text{-a.s.}
$$
\end{theorem}


We next recall a basic integrability fact 
for stationary solutions, which will be used 
later when the mixed third-order fluxes are 
introduced. Related $L^p$--
estimates for stochastic SQG 
were obtained in 
\cite{ref-Zhu-2015-AOP}.
For completeness, we give a direct 
proof of the estimate needed 
in the present additive stationary setting
in Appendix \hyperref[app:main]{A}.

\begin{proposition}
\label{Prop2.4}
Let $\theta$ be a statistically stationary strong solution given by 
Theorem \ref{Thm2.2}.
Then, for every fixed
$\nu,\alpha>0$ and $\lambda\ge1$,
$$
\E\norm{\theta}_{L^3_\lambda}^3<\infty.
$$
Consequently, for every $\ell>0$, the structure functions $S_E(\ell)$ 
and $S_H(\ell)$,
appearing in Theorem \ref{Thm} below,
are finite and hence well defined.
\end{proposition}

\subsection{ SPE and Hamiltonian balance laws}\label{SECTION2.3}
\ \ \ \ The inviscid SQG equation
has two quadratic, 
nonnegative definite conserved quantities: 
the SPE $E(\theta)=
\|\theta\Vert_{L^2}^2$
 and the Hamiltonian $H(\theta)=
\|\Lambda^{-\frac{1}{2}}\theta\Vert_{L^2}^2$.
For statistically stationary solutions of
\eqref{eq:(1.1)},
applying Itô's formula to $E(\theta)$
 and $H(\theta)$, we
 derive the following balance relations: 
\begin{equation}
    \nu\mathbb{E}\|\nabla\theta\Vert_{L_\lambda^2}^2+
\alpha\mathbb{E}\|\Lambda^
{-\gamma}\theta\Vert_{L_\lambda^2}^2=
\varepsilon,
\label{eq:(balance I)}
\end{equation}
\begin{equation}
    \nu\mathbb{E}\|\Lambda^{\frac{1}{2}}
\theta\Vert_{L_\lambda^2}^2+
\alpha\mathbb{E}\|\Lambda^
{-\gamma-{\frac{1}{2}}}\theta\Vert_
{L_\lambda^2}^2=
\eta,
\label{eq:(balance II)}
\end{equation}
where 
$$\varepsilon:=\frac{1}{2}
\sum_{n\in\mathbb{N}}\fint_
{\mathbb{T}_\lambda^2}
|g^\lambda_n(x)\vert^2\,dx<\infty$$
denotes the 
averaged rate of the SPE input, and
$$\eta:=\frac{1}{2}
\sum_{n\in\mathbb{N}}\fint_{\mathbb{T}_\lambda^2}
|\Lambda^{-\frac{1}{2}}g^\lambda_n(x)\vert^2\,dx<\infty$$
denotes the 
averaged rate of the Hamiltonian input.
Here we assume that these averaged 
input rates are 
normalized to be independent of the box size $\lambda$.
We also introduce weighted forcing intensity
$$\chi:=\frac{1}{2}
\sum_{n\in\mathbb{N}}\fint_{\mathbb{T}_\lambda^2}
|\Lambda^{\frac{1}{2}}g^\lambda_n(x)\vert^2\,dx<\infty,$$
which determines the coefficient of the small-scale cubic law for the Hamiltonian structure function in Theorem~\ref{other laws}. As with $\varepsilon$ and $\eta$, we assume that $\chi$ is normalized to be independent of the box size $\lambda$.

\subsection{Weak anomalous dissipation (W.A.D.) assumption}\label{2.3}

\ \ \ \ We now pass from the fixed-parameter stationary problem to the inviscid
large-box regime used in the main results. Let
$\{\theta^{\nu,\alpha}\}_{\nu,\alpha>0}$ be a family of statistically stationary
solutions to \eqref{eq:(1.1)}, posed on
$
\mathbb T^2_{\lambda},\ 
\lambda=\lambda(\nu,\alpha)\to\infty
\ \text{as}\ \nu,\alpha\to0 .
$
The balance laws \eqref{eq:(balance I)}--\eqref{eq:(balance II)} suggest a
natural separation of the two dissipative mechanisms. We 
now introduce the following
weak anomalous dissipation (W.A.D.) assumption,
which requires that, in the vanishing 
viscosity and damping limit, the 
SPE input is asymptotically 
dissipated entirely by the small-scale 
viscosity, while the
Hamiltonian input is asymptotically removed 
entirely by the large-scale damping:
\begin{equation}
\lim_{\nu,\alpha \to 0}\alpha \mathbb{E} \| 
\Lambda^{-\gamma} \theta \|_
{{L}_\lambda^2}^2 = 0,\ \lim_{\nu,\alpha \to 0}
  \nu \mathbb{E} \| \Lambda^{1/2} 
\theta \|_{{L}_\lambda^2}^2 = 0
\label{eq:(1.2)}.
\end{equation}

\begin{remark}
 The W.A.D. \eqref{eq:(1.2)} should be understood as a weak
zeroth-law-type assumption on the stationary family, rather than as an a priori
estimate.
Assumptions of this form have been used to derive rigorous third-order
laws for stationary turbulence, for instance in the Kolmogorov $4/5$ law for
3D Navier--Stokes equations and in the dual-cascade flux laws for 2D
Navier--Stokes equations; see
\cite{ref-2D-NS-dual, ref-3D-NS-2019}. 
The existence of the solutions satisfying
\eqref{eq:(1.2)} remains an open problem, 
just as the corresponding
W.A.D. assumptions remain open in the Navier--Stokes
settings mentioned above.
\end{remark}

\begin{remark}
    When taking $\nu\sim\alpha$, the two conditions in W.A.D.
    \eqref{eq:(1.2)} are equivalent. Indeed, by interpolation and Hölder's inequality, together with $\nu\sim\alpha$,
    we can obtain
    $$\nu\mathbb E\|\Lambda^{1/2}\theta\Vert_{L^2_\lambda}^2
\lesssim
\bigl(\nu\mathbb E\|\Lambda\theta\Vert_{L^2_\lambda}^2\bigr)^{\frac{2\gamma+1}{2\gamma+2}}
\bigl(\alpha\mathbb E\|\Lambda^{-\gamma}\theta\Vert_{L^2_\lambda}^2\bigr)
^{\frac{1}{2\gamma+2}}\leq
\varepsilon^{\frac{2\gamma+1}{2\gamma+2}}
\bigl(\alpha\mathbb E\|\Lambda^{-\gamma}\theta\Vert_{L^2_\lambda}^2\bigr)
^{\frac{1}{2\gamma+2}}.$$
Thus the first condition in W.A.D. implies the second one.
Conversely, by the same argument, 
using the interpolation between $\dot H^{1/2}$ and 
$\dot H^{-\gamma-1/2}$, the second condition 
implies the first one. 

\end{remark}

\begin{remark}
De Rosa and Yuzbasioglu \cite{ref-DeRosaYuzbasioglu26} showed that, on a fixed torus, Hamiltonian compactness rules out anomalous small-scale Hamiltonian dissipation for deterministic generalized SQG in the inviscid limit. In the same spirit, in the
present large-box stationary setting one has the following
frequency non-concentration criterion for W.A.D., obtained by
high--low frequency splitting, interpolation, and the balance laws. For statistically stationary solutions of
\eqref{eq:(1.1)}, suppose that $\nu\sim\alpha$, and 
assume that, for every fixed $K<\infty$ and every fixed $\kappa>0$,
$$\sup_{\nu,\alpha\in(0,1)}\mathbb{E}\|(\theta)_{\leq K}\Vert^2_{L_\lambda^2}<\infty,\quad \sup_{\nu,\alpha\in(0,1)}\mathbb{E}\|(\Lambda^{-1/2}\theta)_{\geq \kappa}\Vert^2_{L_\lambda^2}<\infty,$$
then W.A.D. \eqref{eq:(1.2)} holds. 

In this sense, the above criterion gives a frequency non-concentration viewpoint on W.A.D. adapted to the dual-cascade picture: SPE does not accumulate at any fixed finite range of frequencies, and Hamiltonian does not accumulate away from the zero-frequency region.

\end{remark}

\subsection{Main results}
\ \ \ \ The main results of this paper,
presented below, are the 
flux laws for dual cascade in such
SQG model:

\begin{theorem}
\label{Thm}
Let $\{\theta^{\nu,\alpha}\}_{\nu,\alpha>0}:
=\{\theta\}_{\nu,\alpha>0}$
be a sequence of 
statistically stationary solutions of 
\eqref{eq:(1.1)};
see Section \ref{2.2 Stationary} for 
the precise 
definition and existence result.
Assume that the forcing 
$\{g_n^\lambda\}_{n\in\mathbb N}$ 
satisfies
\eqref{eq:(regularity for noise)}--
\eqref{eq:(Low F)}.
Under W.A.D. \eqref{eq:(1.2)} and
taking
$\nu\sim\alpha$, in the large-box limit 
$\lim_{\nu,\alpha\to0}\lambda(\nu,\alpha)=\infty$, 
we have:

(i)\ Direct\ cascade: In the small-scale 
interval $(\ell_\nu, \ell_I),\
\lim_{\nu\to0}\ell_\nu=0$, the following 
law holds
\begin{equation}
    \lim_{\ell_I\to 0}\limsup _{\nu,\alpha\to0}
\sup _{\ell\in[\ell_\nu,\ell_I]}\left\lvert 
    \frac{S_E(\ell)}{\ell}+2\varepsilon
    \right\rvert=0, 
\end{equation}
where $S_E(\ell)$ is a
third-order mixed structure function associated
with the 
SPE flux, given by
\begin{equation}
S_E(\ell) := \mathbb{E} 
\fint_{\mathbb{S}^1}\fint_{\mathbb{T}_\lambda^2} 
|\delta_{\ell \hat{n}} \theta|^2 
(\delta_{\ell \hat{n}} u \cdot \hat{n} ) \, dx 
\, dS(\hat{n}),
\end{equation}
and $\ell_\nu$ is the dissipative scale which can be chosen as
$$\lim_{\nu,\alpha\to0}\frac{\nu \mathbb{E}\|\theta\|_{L_\lambda^2}^2 } {\ell^2_\nu} =0.$$

(ii)\ Inverse\ cascade: In the large-scale 
interval $(\ell_I, \ell_\alpha),\
\lim_{\alpha\to0}\ell_\alpha=\infty$, the following 
law holds:
\begin{equation}
     \lim_{\ell_I\to \infty}\limsup _{\nu,\alpha\to0}
\sup _{\ell\in[\ell_I,\ell_\alpha]}\left\lvert 
    \frac{S_H(\ell)}{\ell}-\eta
    \right\rvert=0,
\end{equation}
where $S_H(\ell)$ is a
mixed third-order structure 
function associated 
with the Hamiltonian flux, given by
\begin{equation}
S_H(\ell) := \mathbb{E} 
\fint_{\mathbb{S}^1}
\fint_
{\mathbb{T}_\lambda^2} \delta_{\ell\hat{n}}\theta
 \delta_{\ell\hat{n}}\psi
 (\delta_{\ell\hat{n}}u \cdot \hat{n}) \, dx \, dS(\hat{n}),
\end{equation}
and $\ell_\alpha$ is the damping scale which can be chosen as
$$ \lim_{\nu,\alpha\to0}\ell_{\alpha}^2 \left({\alpha} \mathbb{E}\|\Lambda^{-\gamma+\frac{1}{2}} \theta\Vert_{L_\lambda^2}^2\right)=0,\quad \lim_{\nu,\alpha\to0}\frac{\ell_\alpha}{\lambda}=0. $$

\end{theorem}

Although the SPE and Hamiltonian exhibit non-zero fluxes in the direct and inverse cascade ranges, respectively, their third-order structure functions display higher-order or vanishing behavior in the opposite inertial ranges. 
The following theorem makes these complementary scaling behaviors precise.

\begin{theorem}\label{other laws}
    Assume the hypotheses of Theorem~\ref{Thm}, and let $\ell_\nu$ and $\ell_\alpha$ be scales satisfying the conditions stated therein.
Then the following laws hold:

(i) Small-scale Hamiltonian law:\
There exists a dissipative scale 
$\widetilde{\ell}_\nu$ with
 $\lim_{\nu\to0}\widetilde{\ell}_\nu=0$,
 such that
$$
\lim_{\ell_I\to0}
\limsup_{\nu,\alpha\to0}
\sup_{\ell\in[\widetilde{\ell}_\nu,\ell_I]}
\left|
\frac{S_H(\ell)}{\ell^3}
-\frac{\chi}{8}
\right|=0,$$
where $\widetilde{\ell}_\nu$ can be chosen as
$$\lim_{\nu,\alpha\to0}\frac{\nu \mathbb{E}\|
\Lambda^{1/2}\theta\|_{L_\lambda^2}^2 } 
{\widetilde{\ell}^2_\nu} =0,\quad
\widetilde{\ell}_\nu\geq\ell_\nu.$$

(ii) Large-scale SPE law:\
 In the large-scale inertial range of Theorem~\ref{Thm}, we have
$$
\lim_{\ell_I\to\infty}
\limsup_{\nu,\alpha\to0}
\sup_{\ell\in[\ell_I,\ell_\alpha]}
|S_E(\ell)|=0.
$$

\end{theorem}

We also state two Onsager-type obstruction results,
which show that
sufficiently regular stationary 
families cannot support the corresponding non-trivial flux laws.

\begin{theorem}
\label{Onsager I}
Let $\{\theta\}_{\nu,\alpha>0}$ be a family of statistically stationary solutions
to \eqref{eq:(1.1)}.
Assume that for some $s>1/3$,
$$
\sup_{\nu,\alpha}
\mathbb E
\|\theta\|^3_{B^s_{3,\infty}(\mathbb T^2_\lambda)}
<\infty,
$$
where the Besov norm is defined with respect to 
the normalized $L^3_\lambda$ norm.
Then
$$
\lim_{\ell_I\to 0}
\limsup _{\nu,\alpha\to0}
\sup_{\ell\in(0,\ell_I]}
\left|
\frac{S_E(\ell)}{\ell}
\right|
=0.
$$
Consequently, along such a uniformly Besov-regular family, no non-trivial direct SPE
flux law can hold.
\end{theorem}

The direct obstruction above is governed by high-frequency 
small-scale regularity. For the inverse Hamiltonian cascade, 
the relevant obstruction instead comes from
the low-frequency counterpart.
This leads to the following low-frequency Onsager-type obstruction.

\begin{theorem}
\label{Onsager II}
Let $\{\theta\}_{\nu,\alpha>0}$ be a family of statistically 
stationary solutions
to \eqref{eq:(1.1)}.
Suppose that for some $s\in(0,1)$,
$$
\sup_{\nu,\alpha}
\mathbb E
\left[
\|\theta\|^2_{L^3_\lambda}
\left(
\|\theta\|_{L^3_\lambda}
+
\|\theta\|^{\mathrm{low}}_{\dot B^{-s}_{3,\infty}(\mathbb T^2_\lambda)}
\right)
\right]
<\infty.
$$
Then
$$
\lim_{\ell_I\to\infty}
\limsup_{\nu,\alpha\to0}
\sup_{\ell\in[\ell_I,\lambda)}
\left|
\frac{S_H(\ell)}{\ell}
\right|
=0.
$$
Consequently, along such a family that 
is uniformly regular at low frequencies,
no non-trivial inverse
Hamiltonian flux law can hold.
\end{theorem}

\section{\textbf{KHM relation and scaling laws for the SPE}}  
\setcounter{equation}{0}
\label{3}

\ \ \ \ The derivation of the flux laws is based on the 
Kármán--Howarth--Monin (KHM) relation, an exact two-point 
balance identity for stationary solutions. In the present 
forced--dissipative setting, it relates the divergence of a 
third-order flux to the viscous dissipation, large-scale damping, and 
forcing correlation terms.

\begin{proposition}

Let $\theta$ be a stationary 
pathwise solution to the stochastic SQG equation
\eqref{eq:(1.1)}.
Then the following
KHM relation holds:
\begin{equation}
    \nabla\cdot F(h)=-4\nu\Delta\Gamma(h)
    +4\alpha M(h)-4Q(h)
    \label{eq:(1.5)}
\end{equation}
where $\nabla\cdot$ and $\Delta$ are applied to the $h$ variable and 
\begin{align*}
    &F(h)=\mathbb{E}\fint_
{\mathbb{T}_\lambda^2}|\delta _h\theta(x) \vert ^2\delta_hu(x)\,dx,\\
&\Gamma(h)=\mathbb{E}\fint_
{\mathbb{T}_\lambda^2}\theta(x)\theta(x+h)dx,\\
&M(h)=\mathbb{E}\fint_
{\mathbb{T}_\lambda^2}\Lambda^{-\gamma}
\theta(x)\Lambda^{-\gamma}\theta(x+h)\,dx,\\
&Q(h)=\frac{1}{2}\sum_{n\in\mathbb{N}}
\fint_
{\mathbb{T}_\lambda^2}g^\lambda_n(x)
g^\lambda_n(x+h)dx.
\end{align*}
\end{proposition}

\begin{proof}
  \  We denote $\theta(t,x):=
    \theta$ and $\theta(t,x+h):=
    \widetilde{\theta}$ in the following proof
    for brevity, and the same notation holds 
    for $\widetilde{u}$ 
    , $\widetilde{g}^\lambda_n$ and 
    $\widetilde{\psi}$.

    Since $\theta$ and $\widetilde{\theta}$
    satisfy
    \begin{equation}
        \begin{cases}
d\theta =-( u \cdot \nabla \theta
-\nu\Delta\theta+\alpha\Lambda ^{-2\gamma}
\theta)dt +\sum_{n\in\mathbb{N}}g^
\lambda_n dW_n, \\ 
d\widetilde{\theta}  =-( \widetilde{u}
 \cdot \nabla \widetilde{\theta}
-\nu\Delta\widetilde{\theta}
+\alpha\Lambda ^{-2\gamma}
\widetilde{\theta})dt + \sum_
{n\in\mathbb{N}}
\widetilde{g}^\lambda_n dW_n,
        \end{cases}
    \end{equation}
using Itô's product rule
\begin{equation}
    d(\theta\widetilde{\theta})
    =\theta d\widetilde{\theta}+
\widetilde{\theta}d\theta+
d\left\langle \theta,
\widetilde{\theta}\right\rangle_t,
\end{equation}
and note that $\theta$ and $\widetilde
{\theta}$ 
are stationary,
we derive
\begin{align*}
    -&\mathbb{E}\fint_
{\mathbb{T}_\lambda^2}\left[\theta\nabla
\cdot(\widetilde{u}\widetilde{\theta})
+\widetilde{\theta}\cdot\nabla\cdot(u\theta)\right] \,dx
+\nu\mathbb{E}\fint_
{\mathbb{T}_\lambda^2}\left( \theta\Delta\widetilde{\theta}+
\widetilde{\theta}\Delta\theta\right)\,dx\\
&-\alpha\mathbb{E}\fint_
{\mathbb{T}_\lambda^2}\left(\theta\Lambda^{-2\gamma}
\widetilde{\theta}
+\widetilde{\theta}\Lambda^{-2\gamma}\theta\right)\,dx 
+\mathbb{E}\sum_{n\in\mathbb{N}}
\fint_
{\mathbb{T}_\lambda^2}g^\lambda_n\widetilde{g}^\lambda_n\,dx=0.
\end{align*}
Integrating by parts, and 
noting that $-\Delta$ and $\Lambda^{-2\gamma}$
are both self-adjoint, then
$$\mathbb{E}\sum_{k=1}^2\partial_{h_k}
\fint_
{\mathbb{T}_\lambda^2}(u_k\theta\widetilde{\theta}-
\theta\widetilde{u}_k\widetilde{\theta})\,dx
+2\nu\mathbb{E}\Delta\fint_
{\mathbb{T}_\lambda^2}\theta\widetilde{\theta}\,dx
-2\alpha\mathbb{E}\fint_
{\mathbb{T}_\lambda^2}
\Lambda^{-\gamma}
\theta\Lambda^{-\gamma}\widetilde{\theta}\,dx
+\mathbb{E}\sum_{n\in\mathbb{N}}
\fint_
{\mathbb{T}_\lambda^2}g^\lambda_n\widetilde{g}^\lambda_n\,dx=0.
$$
Since $F(h)=\mathbb{E}\fint_
{\mathbb{T}_\lambda^2}|\delta _h\theta(x) 
\vert ^2\delta_hu(x)\,dx$ and 
one can verify that
$$\nabla\cdot F(h)=
2\mathbb{E}\sum_{k=1}^2\partial_{h_k}
\fint_
{\mathbb{T}_\lambda^2}(u_k\theta\widetilde{\theta}-
\theta\widetilde{u}_k\widetilde{\theta})\,dx,$$
we obtain the KHM relation \eqref{eq:(1.5)}.

\end{proof}

\begin{proof}[Proof of Theorem \ref{Thm}-(i)]

Integrating both sides of the KHM relation \eqref{eq:(1.5)} over
$\{|h\vert\leq\ell\}$ and applying the divergence theorem,
we derive
\begin{equation}
    S_E(\ell)=-2\nu\ell\fint_
{\{|h\vert\leq\ell\}}\Delta\Gamma(h)dh+2\alpha\ell\fint_
{\{|h\vert\leq\ell\}}M(h)dh-2\ell\fint_
{\{|h\vert\leq\ell\}}Q(h)dh,
\label{eq:(1.7)}
\end{equation}
where the SPE flux
structure function $S_E(\ell)$
is defined as
\begin{equation*}
    S_E(\ell):=\mathbb{E} 
\fint_{\mathbb{S}^1}\fint_
{\mathbb{T}_\lambda^2}  
|\delta_{\ell \hat{n}}
 \theta|^2 
(\delta_{\ell \hat{n}} u \cdot \hat{n} )
 \, dx \, dS(\hat{n}).
\end{equation*}
Rewrite the right-hand side of \eqref{eq:(1.7)} 
in terms
 of a spherical average, we derive
\begin{equation}
    \frac{S_E(\ell)}{\ell}
    =-\frac{4\nu\bar{\Gamma}'(\ell)}{\ell}+
    \frac{4\alpha}{\ell^2}\int_0^\ell
r\bar{M}(r)dr-\frac{4}{\ell^2}\int_0^\ell
r\bar{Q}(r)dr,
\label{eq:(1.8)}
\end{equation}
where $\bar{\Gamma}(\ell):=
\frac{1}{2\pi}\int_{\mathbb{S} ^1}
\Gamma(\ell\hat{n})\,dS(\hat{n})$,
and the same notation applies to $\bar{M}(\ell)$ 
and
$\bar{Q}(\ell)$.

For the first term 
on the RHS of
 \eqref{eq:(1.8)}, we have
\begin{equation}
\begin{aligned}
    \sup_{\ell\in(\ell_\nu,\ell_I)}
    \frac{\nu}{\ell}
    |\bar{\Gamma}'(\ell)\vert
    &\lesssim \frac{1}{\ell_\nu}
    \left(\nu\mathbb{E}\|\nabla\theta
    \Vert_{L_\lambda^2}^2\right)^
    {\frac{1}{2}}
     \left(\nu\mathbb{E}\|\theta\Vert
     _{L_\lambda^2}^2\right)^
    {\frac{1}{2}}\\   
\end{aligned}
\label{eq:(1.9)}
\end{equation}
{By interpolation and H\"older's inequality,  
\begin{align*}
\nu \mathbb{E}\|\theta\|_{L_\lambda^2}^2 &\lesssim\left(\nu \mathbb{E}\|\Lambda \theta\|_{L_\lambda^2}^2\right)^{\frac{\gamma}{\gamma+1}}\left(\nu \mathbb{E}\|\Lambda^{-\gamma} \theta\Vert_{L_\lambda^2}^2\right)^{\frac{1}{\gamma+1}}\\
&\lesssim\left(\nu \mathbb{E}\|\Lambda \theta\|_{L_\lambda^2}^2\right)^{\frac{\gamma}{\gamma+1}}
\left(\alpha \mathbb{E}\|\Lambda^{-\gamma} \theta\Vert_{L_\lambda^2}^2\right)^{\frac{1}{\gamma+1}}
\end{align*}
as $\nu\sim \alpha$. By the SPE balance and W.A.D.\eqref{eq:(1.2)}, 
$$\nu \mathbb{E}\|\Lambda \theta\|_{L_\lambda^2}^2\leq \varepsilon,\quad \alpha \mathbb{E}\|\Lambda^{-\gamma} \theta\Vert_{L_\lambda^2}^2 \rightarrow 0.$$
Hence $\nu \mathbb{E}\|\theta\|_{L_\lambda^2}^2 \to 0$. So we may choose $\ell_{\nu}$ to be any scale satisfying 
$$\left(\nu \mathbb{E}\|\theta\|_{L_\lambda^2}^2\right)^{\frac{1}{2}} \ll \ell_\nu \ll 1.$$

 For the second term 
associated with damping, by W.A.D.
\eqref{eq:(1.2)} we have
\begin{equation}
\left|\frac{4\alpha}{\ell^2}\int_0^\ell r 
\bar{M}(r)dr\right|
\lesssim\alpha\mathbb{E}\|\Lambda^{-\gamma}
\theta\Vert_{L_\lambda^2}^2\to0.
\label{eq:(1.11)}
\end{equation}

For the last forcing term, 
we need a continuity 
estimate which is uniform in the box 
size.
Indeed, by the Fourier expansion,
$$
Q(h)-Q(0)
=
\frac12
\sum_n
\sum_{k\in \mathbb Z_\lambda^2}
|\widehat g_n^\lambda(k)|^2
\bigl(e^{ik\cdot h}-1\bigr).
$$
For any $R\ge 1$, splitting 
the sum into $|k|\le 
R$ and $|k|>R$, we obtain
$$
|Q(h)-Q(0)|
\lesssim
\varepsilon R|h|
+
R^{-2(1+\sigma)}
\sum_n
\|\Lambda^{1+\sigma}g_n^\lambda\|_{L^2_\lambda}^2 .
$$
Taking the supremum over $\lambda\ge1$, using the 
forcing assumption \eqref{eq:(regularity for noise)},
and then choosing $R=|h|^{-1/2}$, we get
$$
\sup_{\lambda\ge1}
|Q(h)-Q(0)|
\to 0,
\qquad |h|\to0.
$$
Therefore, since $Q(0)=\varepsilon$,
$$
\sup_{\ell\in[\ell_\nu,\ell_I]}
\left|
\frac4{\ell^2}\int_0^\ell r\bar{Q}(r)\,dr
-
2\varepsilon
\right|
\le
2\sup_{|h|\le \ell_I}
|Q(h)-Q(0)|
\to0
$$
as $\ell_I\to0$, uniformly in the inviscid large-box limit.
Thus we derive
\begin{equation}
    \frac{4}{\ell^2}\int_0^\ell r\bar{Q}(r)
    \,dr\to 2\varepsilon,\quad
\text{as}\ \nu,\alpha\to0,\ \ell_I\to0.
    \label{eq:(1.12)}
\end{equation}
Now combining
 \eqref{eq:(1.9)}, 
  \eqref{eq:(1.11)} and
   \eqref{eq:(1.12)},
   we obtain 
\begin{equation*}
    \lim_{\ell_I\to 0}\limsup _{\nu,\alpha\to0}
\sup _{\ell\in[\ell_\nu,\ell_I]}\left\lvert 
    \frac{S_E(\ell)}{\ell}+2\varepsilon
    \right\rvert=0.
\end{equation*}
This completes the proof.
}
\end{proof}

We next use the same KHM relation to 
prove the large-scale vanishing of the SPE structure function.

\begin{proof}[Proof of Theorem~\ref{other laws}-(ii)]
We use \eqref{eq:(1.8)} multiplied by $\ell$.
As in the proof of Theorem~\ref{Thm}-(i),
$$
\nu|\bar{\Gamma}'(\ell)|
\leq
\varepsilon^{1/2}
\left(\nu\mathbb{E}\|\theta\|_{L_\lambda^2}^2\right)^{1/2}
\to0,
$$
uniformly in $\ell>0$.

For the damping term, interpolation and the Hamiltonian balance
\eqref{eq:(balance II)} give
\begin{align*}
\left|\frac{4\alpha}{\ell}
\int_0^\ell r\bar{M}(r)\,dr\right|
&\leq
2\ell\alpha\mathbb{E}
\|\Lambda^{-\gamma}\theta\|_{L_\lambda^2}^2\\
&\leq
2\ell
\left(\alpha\mathbb{E}
\|\Lambda^{-\gamma-1/2}\theta\|_{L_\lambda^2}^2\right)^{1/2}
\left(\alpha\mathbb{E}
\|\Lambda^{-\gamma+1/2}\theta\|_{L_\lambda^2}^2\right)^{1/2}\\
&\leq
2\eta^{1/2}
\left(\ell_\alpha^2\alpha\mathbb{E}
\|\Lambda^{-\gamma+1/2}\theta\|_{L_\lambda^2}^2\right)^{1/2}
\to0,
\end{align*}
uniformly for $\ell\in[\ell_I,\ell_\alpha]$, by the choice of
$\ell_\alpha$ in Theorem~\ref{Thm}.

For the forcing term, Fourier expansion 
yields
$$
Q(h)=\frac{1}{2}\sum_{n}
\sum_{k\in\mathbb{Z}^2_\lambda\setminus\{0\}}
|\hat{g}^\lambda_n(k)|^2e^{ik\cdot h},
$$
then
$$
\bar Q(r)
=\frac12\sum_n\sum_{k\in\mathbb{Z}^2_\lambda\setminus\{0\}}
|\widehat g_n^\lambda(k)|^2J_0(r|k|),
$$
where
$
J_0(z)=\frac1{2\pi}\int_0^{2\pi}
e^{iz\cos\theta}\,d\theta
$ denotes 
the Bessel function of the first kind of order zero.
Then
integrating in polar
coordinates yields
$$
\frac{4}{\ell}\int_0^\ell r\bar{Q}(r)\,dr
=
2\sum_n\sum_{k\in\mathbb{Z}^2_\lambda\setminus\{0\}}
\frac{|\widehat g_n^\lambda(k)|^2}{|k|}
J_1(\ell|k|),
$$
where $J_1$ is the Bessel function of the first kind of order one, given by
$$
\int_0^\ell rJ_0(r|k|)\,dr
=\frac{\ell}{|k|}J_1(\ell|k|).
$$
By the standard asymptotics of the Bessel function (see e.g. \cite{ref-Bessel}), we have
$|J_1(z)|\lesssim\min\{1,z^{-1/2}\}$ for $z>0$.
Then splitting the sum at $|k|=\delta$, we obtain
$$
\left|\frac{4}{\ell}\int_0^\ell r\bar{Q}(r)\,dr\right|
\lesssim
\sum_n
\|(\Lambda^{-1/2}g_n^\lambda)_{\leq\delta}\|_{L_\lambda^2}^2
+
\eta(\ell\delta)^{-1/2}.
$$
Choosing $\delta=\ell_I^{-1/2}$, we get
$$
\sup_{\lambda\geq1}\sup_{\ell\geq\ell_I}
\left|\frac{4}{\ell}\int_0^\ell r\bar{Q}(r)\,dr\right|
\lesssim
\sup_{\lambda\geq1}\sum_n
\|(\Lambda^{-1/2}g_n^\lambda)_{\leq\ell_I^{-1/2}}\|_{L_\lambda^2}^2
+
\eta\ell_I^{-1/4}
\to0
$$
as $\ell_I\to\infty$, by \eqref{eq:(Low F)}.
Combining the three estimates with \eqref{eq:(1.8)} and taking
$\limsup_{\nu,\alpha\to0}$ followed by $\ell_I\to\infty$
completes the proof.
\end{proof}

\

\section{\textbf{KHM relation and flux laws for the Hamiltonian}}  
\setcounter{equation}{0}
\label{4}

We now derive the analogous KHM 
relation associated with the Hamiltonian balance.

\begin{proposition}

Let $\theta$ be a stationary 
pathwise solution to SQG \eqref{eq:(1.1)}.
Then the following
KHM relation holds:
\begin{equation}
    \nabla\cdot J(h)=-2\nu A(h)+
    2\alpha B(h)-2K(h)
    \label{eq:(Inverse KHM)}
\end{equation}
where
\begin{align*}
    &J(h)= \mathbb{E} 
\fint_
{\mathbb{T}_\lambda^2} \delta_{h}\theta
 \delta_{h}\psi
\delta_{h}u  \, dx \\
&A(h)=-
\mathbb{E}\fint_
{\mathbb{T}_\lambda^2}\Lambda^{\frac{1}{2}}\theta(x)
\Lambda^{\frac{1}{2}}\theta(x+h)\,dx,\\
&B(h)=\mathbb{E}\fint_
{\mathbb{T}_\lambda^2}
\Lambda^{-\gamma-\frac{1}{2}}
\theta(x)\Lambda^{-\gamma-\frac{1}{2}}
\theta(x+h)\,dx,\\
&K(h)=
\frac{1}{2}\sum_{n\in\mathbb{N}}\fint_
{\mathbb{T}_\lambda^2}\Lambda^{-\frac{1}{2}}
g^\lambda_n(x)\Lambda^{-\frac{1}{2}}
g^\lambda_n(x+h)dx.
\end{align*}
\end{proposition}

\begin{proof}
Since $\psi=\Lambda^{-1}\theta$, applying $\Lambda^{-1}$ to
 \eqref{eq:(1.1)} yields
\begin{equation}
    d\psi + \bigl(\Lambda^{-1}(u\cdot \nabla \theta)-\nu \Delta \psi+\alpha \Lambda^{-2\gamma}\psi\bigr)\,dt
=
\sum_{n\in \mathbb{N}}\Lambda^{-1}g_n^\lambda(x)\,dW_n(t).
\end{equation}
Hence $\theta$ and $\widetilde{\psi}$ satisfy
\begin{equation}
\left\{
\begin{aligned}
d\theta
&=
\bigl(-u\cdot \nabla\theta+\nu\Delta\theta-\alpha\Lambda^{-2\gamma}\theta\bigr)\,dt
+\sum_{n\in\mathbb N} g_n^\lambda\,dW_n,\\
d\widetilde{\psi}
&=
\bigl(-\Lambda^{-1}(\widetilde{u}\cdot\nabla \widetilde{\theta})
+\nu\Delta\widetilde{\psi}
-\alpha\Lambda^{-2\gamma}\widetilde{\psi}\bigr)\,dt
+\sum_{n\in\mathbb N}\Lambda^{-1}\widetilde{g}_n^\lambda\,dW_n.
\end{aligned}
\right.
\end{equation}
Applying It\^o's product rule to $\theta\widetilde{\psi}$, we obtain
$$
d\bigl(\theta \widetilde{\psi}\bigr)
=
\widetilde{\psi}\,d\theta
+\theta\,d\widetilde{\psi}
+d\langle \theta,\widetilde{\psi}\rangle_t.
$$
Taking expectation, integrating over $\mathbb{T}_\lambda^2$, and using stationarity, we get
$$
\begin{aligned}
-& \mathbb{E}\fint_
{\mathbb{T}_\lambda^2}
\Bigl[
\widetilde{\psi}\,u\cdot\nabla\theta
+
\theta\,\Lambda^{-1}(\widetilde{u}
\cdot\nabla\widetilde{\theta})
\Bigr]\,dx 
+ \nu  \mathbb{E}\fint_
{\mathbb{T}_\lambda^2}
\Bigl(
\widetilde{\psi}\,\Delta\theta
+
\theta\,\Delta\widetilde{\psi}
\Bigr)\,dx \\
&\quad
-\alpha  \mathbb{E}\fint_
{\mathbb{T}_\lambda^2}
\Bigl(
\widetilde{\psi}\,\Lambda^{-2\gamma}\theta
+
\theta\,\Lambda^{-2\gamma}\widetilde{\psi}
\Bigr)\,dx 
+  \mathbb{E}\sum_{n\in\mathbb N}\fint_
{\mathbb{T}_\lambda^2}
g_n^\lambda\Lambda^{-1}\widetilde{g}_n^\lambda\,dx =0.
\end{aligned}
$$
For the transport term, since $\Lambda^{-1}$ is self-adjoint and 
by translation invariance, we have
\begin{equation}
\begin{aligned}
-\mathbb{E}\fint_
{\mathbb{T}_\lambda^2}\left[\widetilde{\psi}u\cdot\nabla\theta+
\theta\Lambda^{-1}(\widetilde{u}
\cdot\nabla\widetilde{\theta})\right]\,dx
&=-\mathbb{E}\fint_
{\mathbb{T}_\lambda^2}\left[\widetilde{\psi}\nabla\cdot(u\theta)+
\psi\nabla\cdot(\widetilde{u}
\widetilde{\theta})\right]\,dx\\
&=\mathbb{E}\fint_
{\mathbb{T}_\lambda^2}\left[\nabla\widetilde{\psi}\cdot(u\theta)+
\nabla\psi\cdot(\widetilde{u}
\widetilde{\theta})\right]\,dx\\
&=\mathbb{E}\fint_
{\mathbb{T}_\lambda^2}\Bigl[u(x)\theta(x)\cdot\nabla\psi(x+h)+
u(x+h)\theta(x+h)\cdot\nabla\psi(x)\Bigr]\,dx\\
&=\mathbb{E}\fint_
{\mathbb{T}_\lambda^2}u(x)\cdot\Bigl[\nabla_h\psi(x+h)-
\nabla_h\psi(x-h)\Bigr]\theta(x)\,dx\\
&=\nabla\cdot\mathbb{E}\fint_
{\mathbb{T}_\lambda^2}\bigg(\theta u \widetilde{\psi}-\widetilde{\theta}\widetilde{u}\psi\bigg)\,dx.
\label{eq:(J(h))}
\end{aligned}
\end{equation}

To identify this expression with $\nabla_h\cdot J(h)$, we expand the three
increments and cancel the same-point terms by translation invariance:
\begin{align*}
J(h)
=\mathbb{E}\fint_{\mathbb{T}_\lambda^2}
\bigg(\theta\widetilde{\psi}u
-\widetilde{\theta}\psi\widetilde{u}\bigg)dx
+\mathbb{E}\fint_{\mathbb{T}_\lambda^2}
\bigg(\widetilde{\theta}\psi u
-\theta\widetilde{\psi}\widetilde{u}\bigg)dx
+\mathbb{E}\fint_{\mathbb{T}_\lambda^2}
\bigg(\theta\psi\widetilde{u}
-\widetilde{\theta}\widetilde{\psi}u\bigg)dx.
\end{align*}
Since $u=\nabla^\perp\psi$, we have
$\nabla\cdot u=0$ and $u\cdot\nabla\psi=0$.
The last four terms have zero $h$-divergence. Indeed,
\begin{align*}
\nabla_h\cdot\mathbb{E}\fint_{\mathbb{T}_\lambda^2}
\widetilde{\theta}\psi u\,dx
&=-\mathbb{E}\fint_{\mathbb{T}_\lambda^2}
\widetilde{\theta}\nabla\cdot(\psi u)\,dx=0,\\
\nabla_h\cdot\mathbb{E}\fint_{\mathbb{T}_\lambda^2}
\theta\widetilde{\psi}\widetilde{u}\,dx
&=\mathbb{E}\fint_{\mathbb{T}_\lambda^2}
\theta\nabla\cdot(\widetilde{\psi}\widetilde{u})\,dx=0,\\
\nabla_h\cdot\mathbb{E}\fint_{\mathbb{T}_\lambda^2}
\theta\psi\widetilde{u}\,dx
&=\mathbb{E}\fint_{\mathbb{T}_\lambda^2}
\theta\psi\nabla\cdot\widetilde{u}\,dx=0,\\
\nabla_h\cdot\mathbb{E}\fint_{\mathbb{T}_\lambda^2}
\widetilde{\theta}\widetilde{\psi}u\,dx
&=-\mathbb{E}\fint_{\mathbb{T}_\lambda^2}
\widetilde{\theta}\widetilde{\psi}\nabla\cdot u\,dx=0.
\end{align*}
Thus the transport term in \eqref{eq:(J(h))} equals
$\nabla_h\cdot J(h)$.
The treatment of the remaining three terms is straightforward, using the self-adjointness of $\Lambda$, and we omit the details.
\end{proof}

\begin{proof}[Proof of Theorem \ref{Thm}-(ii)]

Integrating both sides of the 
KHM relation \eqref{eq:(Inverse KHM)}
 over $\{|h\vert\leq\ell\}$ 
and applying the divergence theorem, we derive
\begin{equation}
    S_H(\ell)=
    -\nu\ell\fint_
{\{|h\vert\leq\ell\}}A(h)dh+\alpha
\ell\fint_
{\{|h\vert\leq\ell\}}B(h)dh-\ell\fint_
{\{|h\vert\leq\ell\}}K(h)dh,
\label{eq:(1.14)}
\end{equation}
where $S_H(\ell)$ denotes the Hamiltonian 
flux structure function 
\begin{equation*}
S_H(\ell) := \mathbb{E} 
\fint_{\mathbb{S}^1}
\fint_
{\mathbb{T}_\lambda^2} \delta_{\ell\hat{n}}\theta
 \delta_{\ell\hat{n}}\psi
 (\delta_{\ell\hat{n}}u \cdot \hat{n}) \, dx \, dS(\hat{n}).
\end{equation*}
Rewriting \eqref{eq:(1.14)} in terms of spherical averages, we obtain 
\begin{equation}
    \frac{S_H(\ell)}{\ell}
    =-\frac{2\nu}{\ell^2}\int
    _0^\ell r \bar{A}(r)dr+
    \frac{2\alpha}{\ell^2}\int_0^\ell
r\bar{B}(r)dr-\frac{2}{\ell^2}\int_0^\ell
r\bar{K}(r)dr,
\label{eq:(1.15)}
\end{equation}
where $\bar{A}(r) = \fint_{\mathbb{S}^1}A(r\hat{n})dS(\hat{n})$ and the same notation applies to $\bar{B}$ and $\bar{K}$. 

On the RHS of \eqref{eq:(1.15)}, by
W.A.D.
\eqref{eq:(1.2)},
the first term tends to zero as $\nu\to0$,
\begin{equation}
   \left| \frac{2\nu}{\ell^2}\int
    _0^\ell r \bar{A}(r)dr\right|
    \lesssim
        \nu\mathbb{E}\|\Lambda^\frac{1}{2}
        \theta\Vert_{L_\lambda^2}^2
     \to0.
\end{equation}

For the second term, we have 
\begin{align*}
     \left|\frac{2\alpha}{\ell^2}\int_0^\ell
r\left(\bar{B}(r)-\bar{B}(0)\right)
\,dr\right|
&=
 \frac{2\alpha}{\ell^2}
 \int_0^\ell
r\mathbb{E}\left\lvert \fint_{\mathbb{S}^1}
\fint_
{\mathbb{T}_\lambda^2}
\Lambda^{-\gamma-\frac{1}{2}}
\theta(x)\cdot\Lambda^{-\gamma-\frac{1}{2}}
[\theta(x+r\hat{n})-\theta(x)]\,dS(\hat{n})
\,dx\right\rvert \,dr\\
&\lesssim
\frac{2\alpha}{\ell^2}
 \int_0^\ell r^2
 \left(\mathbb{E}\|\Lambda^{-\gamma+\frac{1}{2}}\theta
 \Vert_{L_\lambda^2}^2\right)^\frac{1}{2}
 \left(\mathbb{E}\|\Lambda^{-\gamma-\frac{1}{2}}
 \theta\Vert_{L_\lambda^2}^2\right)^\frac{1}{2}\,dr\\
 &\lesssim 
  \ell_{\alpha}\left({\alpha} \mathbb{E}\|\Lambda^{-\gamma+\frac{1}{2}} \theta\Vert_{L_\lambda^2}^2\right)^{\frac{1}{2}} 
  \left({\alpha} \mathbb{E}\|\Lambda^{-\gamma-\frac{1}{2}} \theta\Vert_{L_\lambda^2}^2\right)^{\frac{1}{2}}.
\end{align*}
By interpolation and $\nu\sim\alpha$,
\begin{align*}
\alpha \mathbb{E}\|\Lambda^{-\gamma+\frac{1}{2}} \theta\Vert_{L_\lambda^2}^2 \lesssim\left(\nu \mathbb{E}\|\Lambda^{\frac{1}{2}} \theta\Vert_{L_\lambda^2}^2\right)^{\frac{1}{\gamma+1}}\left(\alpha \mathbb{E}\|\Lambda^{-\gamma-\frac{1}{2}} \theta\Vert_{L_\lambda^2}^2\right)^{\frac{\gamma}{\gamma+1}}.
\end{align*}
By the Hamiltonian balance and W.A.D. \eqref{eq:(1.2)}, 
$$\alpha \mathbb{E}\|\Lambda^{-\gamma-\frac{1}{2}} \theta\Vert_{L_\lambda^2}^2 \leq \eta,\quad
\nu \mathbb{E}\|\Lambda^{\frac{1}{2}} \theta\Vert_{L_\lambda^2}^2 \rightarrow 0,$$
implying 
$$\alpha \mathbb{E}\|\Lambda^{-\gamma+\frac{1}{2}} \theta\Vert_{L_\lambda^2}^2 \to0.$$
Therefore, we may choose the damping scale to satisfy $\ell_{\alpha}/ \lambda \rightarrow 0$, which ensuring the box size to grow compatibly with the damping scale, and thus 
$$ 1 \ll \ell_{\alpha} \ll \min\{\lambda,\left({\alpha} \mathbb{E}\|\Lambda^{-\gamma+\frac{1}{2}} \theta\Vert_{L_\lambda^2}^2\right)^{-\frac{1}{2}}\},  $$
Thus the second term tends to $\eta$ since
$$\frac{2\alpha}{\ell^2}\int_0^\ell
r\bar{B}(0)\,dr
=
\alpha\mathbb{E}\|\Lambda^
{-\gamma-{\frac{1}{2}}}\theta
\Vert_{L_\lambda^2}^2=
\eta-\nu\mathbb{E}\|\Lambda^{\frac{1}{2}}
\theta\Vert_{L_\lambda^2}^2\to\eta
$$
by W.A.D. \eqref{eq:(1.2)}.

For the last term,
we set
$$
f_n^\lambda:=\Lambda^{-1/2}g^\lambda_n.
$$
Then
$$
K(h)=\frac12\sum_{n\in\mathbb N}\fint_{\mathbb T_\lambda^2}
f_n^\lambda(x)f_n^\lambda(x+h)\,dx.
$$
By expanding in Fourier series
and integrating
in polar coordinates,
as in the proof of Theorem~2.9-(ii),
we obtain
$$
\frac{2}{\ell^2}\int_0^\ell r\bar K(r)\,dr
=
\frac12\sum_n\sum_{k\in\mathbb Z_\lambda^2\setminus\{0\}}
\Phi_\ell(k)|\widehat f_n^\lambda(k)|^2,
$$
where
$$
\Phi_\ell(k):=\frac1{\pi\ell^2}\int_{|h|\le \ell}
e^{ik\cdot h}\,dh
=2\frac{J_1(\ell|k|)}{\ell|k|},
$$
Then, by the standard asymptotics of the 
Bessel function,
$$
|\Phi_\ell(k)|\lesssim \min\left\{1,(\ell|k|)^{-3/2}\right\}.
$$
Thus, for every $\delta>0$,
$$
\left|\frac{2}{\ell^2}\int_0^\ell r\bar K(r)\,dr\right|
\lesssim
\frac12\sum_n\|(f_n^\lambda)_{\leq \delta}\|_{L^2_\lambda}^2
+
C(\ell\delta)^{-3/2}\sum_n\|f_n^\lambda\|_{L^2_\lambda}^2.
$$
Since
$$
\sum_n\|f_n^\lambda\|_{L^2_\lambda}^2
=
\sum_n\|\Lambda^{-1/2}g_n^\lambda\|_{L^2_\lambda}^2
=2\eta<\infty,
$$
we obtain
$$
\left|\frac{2}{\ell^2}\int_0^\ell r\bar K(r)\,dr\right|
\lesssim
\frac12\sum_n\|(f_n^\lambda)_{\leq \delta}\|_{L^2_\lambda}^2
+
C\eta(\ell\delta)^{-3/2}.
$$
Choosing $\delta=\ell_I^{-1/2}$ and taking the supremum over $\ell\in[\ell_I,\ell_\alpha]$, we get
$$
\sup_{\ell\in[\ell_I,\ell_\alpha]}
\left|\frac{2}{\ell^2}\int_0^\ell r\bar K(r)\,dr\right|
\lesssim
\frac12\sup_{\lambda\ge1}\sum_n\|(f_n^\lambda)
_{\leq \ell_I^{-1/2}}\|_{L^2_\lambda}^2
+
C\eta\ell_I^{-3/4}.
$$
Therefore, under the low-frequency condition \eqref{eq:(Low F)}
$$
\lim_{\delta\to0}\sup_{\lambda\ge1}\sum_n
\|(\Lambda^{-1/2}g^\lambda_n)_{\leq \delta}
\|_{L^2_\lambda}^2=0,
$$
it follows that
$$
\left|\frac{2}{\ell^2}\int_0^\ell r\bar K(r)\,dr\right|\to0,\quad
\text{as}\ \nu,\alpha\to0,\ \ell_I\to\infty.
$$
Combining the estimates above, 
we obtain
$$
     \lim_{\ell_I\to \infty}\limsup _{\nu,\alpha\to0}
\sup _{\ell\in[\ell_I,\ell_\alpha]}\left\lvert 
    \frac{S_H(\ell)}{\ell}-\eta
    \right\rvert=0.
$$
This completes the proof.
\end{proof}

We next consider the small-scale behavior of the Hamiltonian structure function.

\begin{proof}[Proof of Theorem~\ref{other laws}-(i)]
By W.A.D.\ \eqref{eq:(1.2)}, we may choose
$\widetilde{\ell}_\nu\geq\ell_\nu$ such that
$\widetilde{\ell}_\nu\to0$ and
$$
\frac{\nu\mathbb{E}
\|\Lambda^{1/2}\theta\|_{L_\lambda^2}^2}
{\widetilde{\ell}_\nu^2}
\to0.
$$
The Hamiltonian balance \eqref{eq:(balance II)} gives
$-\nu\bar{A}(0)+\alpha\bar{B}(0)=\eta$.
Hence \eqref{eq:(1.15)} can be rewritten as
\begin{align*}
\frac{S_H(\ell)}{\ell^3}
&=
-\frac{2\nu}{\ell^4}\int_0^\ell
r\bigl(\bar{A}(r)-\bar{A}(0)\bigr)\,dr\\
&\quad+
\frac{2\alpha}{\ell^4}\int_0^\ell
r\bigl(\bar{B}(r)-\bar{B}(0)\bigr)\,dr
-
\frac{2}{\ell^4}\int_0^\ell
r\bigl(\bar{K}(r)-\eta\bigr)\,dr.
\end{align*}

For the viscous term, the definition of $A$ gives
$$
\left|\frac{2\nu}{\ell^4}\int_0^\ell
r\bigl(\bar{A}(r)-\bar{A}(0)\bigr)\,dr\right|
\leq
\frac{2\nu}{\ell^2}\mathbb{E}
\|\Lambda^{1/2}\theta\|_{L_\lambda^2}^2.
$$
For the damping term, translation invariance yields
\begin{align*}
0\leq\bar{B}(0)-\bar{B}(r)
&=
\frac12\mathbb{E}\fint_{\mathbb{S}^1}
\|\delta_{r\hat n}\Lambda^{-\gamma-1/2}\theta\|_{L_\lambda^2}^2
\,dS(\hat n)\\
&\lesssim
r^2\mathbb{E}
\|\Lambda^{-\gamma+1/2}\theta\|_{L_\lambda^2}^2.
\end{align*}
Consequently,
$$
\left|\frac{2\alpha}{\ell^4}\int_0^\ell
r\bigl(\bar{B}(r)-\bar{B}(0)\bigr)\,dr\right|
\lesssim
\alpha\mathbb{E}
\|\Lambda^{-\gamma+1/2}\theta\|_{L_\lambda^2}^2
\to0,
$$
as shown in the proof of Theorem~2.8-(ii).

For the forcing term, recall the Taylor expansion
$$
J_0(z)=1-\frac{z^2}{4}+\frac{z^4}{64}
+\mathcal{O}(z^6),\qquad z\to0,
$$
which implies the global bound
$$
\left|J_0(z)-1+\frac{z^2}{4}\right|
\lesssim z^4,\qquad z\geq0.
$$
Together with the Fourier representation of $\bar K$,
that is
$$
\bar K(r)
=
\frac12\sum_n\sum_{k\in\mathbb Z_\lambda^2\setminus\{0\}}
\frac{|\widehat g_n^\lambda(k)|^2}{|k|}
J_0(r|k|),$$
we obtain
$$
\left|\bar{K}(r)-\eta+\frac{\chi}{4}r^2\right|
\lesssim
r^4\sum_n
\|\Lambda^{3/2}g_n^\lambda\|_{L_\lambda^2}^2
\lesssim r^4,
\qquad 0\leq r\leq1,
$$
where the last bound follows from \eqref{eq:(regularity for noise)}
and is uniform in $\lambda$.
Therefore,
$$
-\frac{2}{\ell^4}\int_0^\ell
r\bigl(\bar{K}(r)-\eta\bigr)\,dr
=
\frac{\chi}{8}+\mathcal{O}(\ell^2).
$$
Combining the estimates above, for $\ell_I\leq1$ we obtain
$$
\sup_{\ell\in[\widetilde{\ell}_\nu,\ell_I]}
\left|\frac{S_H(\ell)}{\ell^3}-\frac{\chi}{8}\right|
\lesssim
\frac{\nu\mathbb{E}
\|\Lambda^{1/2}\theta\|_{L_\lambda^2}^2}
{\widetilde{\ell}_\nu^2}
+
\alpha\mathbb{E}
\|\Lambda^{-\gamma+1/2}\theta\|_{L_\lambda^2}^2
+
\ell_I^2.
$$
Taking $\limsup_{\nu,\alpha\to0}$ and then letting $\ell_I\to0$
completes the proof.
\end{proof}

\section{Onsager-type obstructions}
\label{5}

We first establish the obstruction for the direct SPE flux.

\begin{proof}[Proof of Theorem \ref{Onsager I}]
By the definition of 
$S_E(\ell)$, Hölder’s
inequality gives
$$
|S_E(\ell)|
\leq
C\mathbb E
\int_{\mathbb S^1}
\|\delta_{\ell \hat n}\theta\|^2_{L^3_\lambda}
\|\delta_{\ell \hat n}u\|_{L^3_\lambda}
\,dS(\hat n).
$$
Since $u=R^\perp\theta$, the Riesz transforms 
can commute with
translations and are bounded on $L^3_\lambda$. Hence
$$
\|\delta_{\ell \hat n}u\|_{L^3_\lambda}
\leq
C\|\delta_{\ell \hat n}\theta\|_{L^3_\lambda}.
$$
Therefore,
$$
|S_E(\ell)|
\leq
C\mathbb E
\sup_{\hat n\in\mathbb S^1}
\|\delta_{\ell \hat n}\theta\|^3_{L^3_\lambda}.
$$
The Besov assumption implies, for $0<\ell\leq 1$ 
and $s\in(1/3,1)$
$$
\|\delta_{\ell \hat n}\theta\|_{L^3_\lambda}
\leq
C\ell^s
\|\theta\|_{B^s_{3,\infty}
(\mathbb T^2_\lambda)}.
$$
As for $s\geq 1$, choosing $\tilde{s}\in(1/3,1)$ and using
Besov embedding $B_{3,\infty}^{s} \hookrightarrow
B_{3,\infty}^{\tilde{s}},$
we can derive the analogous
estimate with $\ell^{\tilde{s}}$:

$$
\|\delta_{\ell \hat n}\theta\|_{L^3_\lambda}
\leq
C{\ell^{\tilde{s}}}
\|\theta\|_{B^{\tilde{s}}_{3,\infty}
(\mathbb T^2_\lambda)}\lesssim
C{\ell^{\tilde{s}}}
\|\theta\|_{B^s_{3,\infty}
(\mathbb T^2_\lambda)},\quad 0<\ell\leq 1.
$$

\noindent Combining both cases above, 
choosing $\tilde{s}\in(1/3,\min\{s,1\})$, we have
$$|S_E(\ell)|
\leq
C\ell^{3\tilde{s}}
\mathbb E
\|\theta\|^3_{B^s_{3,\infty}
(\mathbb T^2_\lambda)},\quad
s>1/3.
$$
Taking the supremum over $\nu,\alpha$, we obtain
$$
\sup_{\nu,\alpha}
\left|
\frac{S_E(\ell)}{\ell}
\right|
\leq
C\ell^{3\tilde{s}-1}
\sup_{\nu,\alpha}
\mathbb E
\|\theta\|^3_{B^s_{3,\infty}
(\mathbb T^2_\lambda)}
\lesssim\ell^{3\tilde{s}-1}.
$$
Since $\tilde{s}\in(1/3,\min\{s,1\}),\ s>1/3$, 
the right-hand side tends to zero as 
$\ell\to0$. This proves the
claim.
\end{proof}

To prove the inverse Hamiltonian obstruction, 
we first establish a large-scale increment estimate for 
the stream function.

\begin{lemma}\label{LEMMA5.1}
Let $0<s<1$ and let $\psi=\Lambda^{-1}\theta$. Then, for every $|h|\geq 1$,
$$
\|\delta_h\psi\|_{L^3_\lambda}
\leq
C |h|^{1-s}
\left(
\|\theta\|_{L^3_\lambda}
+
\|\theta\|^{\mathrm{low}}_{\dot B^{-s}_{3,\infty}(\mathbb T^2_\lambda)}
\right),
$$
where the constant $C$ is independent of $\lambda$.
\end{lemma}

\begin{proof}

Since $\theta$ has zero average on $\mathbb{T}^2_\lambda$, 
the homogeneous
Littlewood--Paley decomposition gives
$$
\theta=\sum_{j\in\mathbb Z}\dot\Delta_j\theta .
$$
Consequently, with $\psi=\Lambda^{-1}\theta$,
$$
\psi=\sum_{j\in\mathbb Z}\Lambda^{-1}\dot\Delta_j\theta .
$$
Since translations commute with Fourier multipliers, we have
$$
\delta_h\psi
=
\sum_{j\in\mathbb Z}
\delta_h\Lambda^{-1}\dot\Delta_j\theta .
$$
Hence, by the triangle inequality,
$$
\|\delta_h\psi\|_{L^3_\lambda}
\le
\sum_{j\in\mathbb Z}
\|\delta_h\Lambda^{-1}\dot\Delta_j\theta\|
_{L^3_\lambda}.
$$
Let $\mathcal{J}=\mathcal{J}(h)\le0$ be chosen so that
$$
2^\mathcal{J}\le |h|^{-1}<2^{\mathcal{J}+1}.
$$
For each dyadic block, we use the localized multiplier estimate
$$
\|\delta_h\Lambda^{-1}\dot\Delta_j\theta\|_{L^3_\lambda}
\le C\min\{|h|,2^{-j}\}\|\dot\Delta_j\theta\|_{L^3_\lambda}.
$$
Indeed,
$$
\|\delta_h\Lambda^{-1}\dot\Delta_j\theta\|_{L^3_\lambda}
\le |h|\|\nabla\Lambda^{-1}\dot\Delta_j\theta\|_{L^3_\lambda}
\lesssim |h|\|\dot\Delta_j\theta\|_{L^3_\lambda},
$$
and for the 
second bound,
note that 
$2^j\Lambda^{-1}\dot\Delta_j$ is a uniformly 
$L^3_\lambda$-bounded localized Fourier multiplier, and hence
$$
\|\delta_h\Lambda^{-1}\dot\Delta_j\theta\|_{L^3_\lambda}
\lesssim \|\Lambda^{-1}\dot\Delta_j\theta\|_{L^3_\lambda}
\lesssim 2^{-j}\|\dot\Delta_j\theta\|_{L^3_\lambda}.
$$
Taking the minimum of the two bounds gives the claim.

Now we split the sum
$\sum_{j\in\mathbb Z}
\|\delta_h\Lambda^{-1}\dot\Delta_j\theta\|
_{L^3_\lambda}$
into the three ranges
$$
j\le \mathcal{J},\qquad \mathcal{J}<j\le0,\qquad j>0.
$$
For $j\leq \mathcal{J}$, using the low-frequency Besov seminorm,
$$
\sum_{j\leq \mathcal{J}}
\|\delta_h\Lambda^{-1}\dot\Delta_j\theta\|_{L^3_\lambda}
\leq
C |h|
\sum_{j\leq \mathcal{J}}
2^{s j}
\|\theta\|^{\mathrm{low}}_{\dot B^{-s}_{3,\infty}}
\leq
C |h|^{1-s}
\|\theta\|^{\mathrm{low}}_{\dot B^{-s}_{3,\infty}}.
$$
For $\mathcal{J}<j\leq0$,
$$
\sum_{\mathcal{J}<j\leq0}
\|\delta_h\Lambda^{-1}\dot\Delta_j\theta\|_{L^3_\lambda}
\leq
C
\sum_{\mathcal{J}<j\leq0}
2^{-j}2^{s j}
\|\theta\|^{\mathrm{low}}_{\dot B^{-s}_{3,\infty}}
\leq
C |h|^{1-s}
\|\theta\|^{\mathrm{low}}_{\dot B^{-s}_{3,\infty}}.
$$
For $j>0$, using the $L^3$ boundedness of Littlewood--Paley projections,
$$
\sum_{j>0}
\|\delta_h\Lambda^{-1}\dot\Delta_j\theta\|_{L^3_\lambda}
\leq
C
\sum_{j>0}
2^{-j}
\|\dot\Delta_j\theta\|_{L^3_\lambda}
\leq
C\|\theta\|_{L^3_\lambda}.
$$
Since $|h|\geq1$ and $0<s<1$, this last term is bounded by
$C|h|^{1-s}\|\theta\|_{L^3_\lambda}$. Combining the three estimates gives the
claim.
\end{proof}

\begin{proof}[Proof of Theorem \ref{Onsager II}]

By the definition of $S_H(\ell)$ and Hölder's inequality,
$$
|S_H(\ell)|
\leq
C\mathbb E
\int_{\mathbb S^1}
\|\delta_{\ell\hat{n}}\theta\|_{L^3_\lambda}
\|\delta_{\ell\hat{n}}\psi\|_{L^3_\lambda}
\|\delta_{\ell\hat{n}} u\|_{L^3_\lambda}
\,dS(\hat n).
$$
Since $u=
R^\perp\theta$ and the Riesz transforms are bounded on
$L^3_\lambda$,
$
\|u\|_{L^3_\lambda}
\leq
C\|\theta\|_{L^3_\lambda}.
$
Moreover,
applying Lemma~\ref{LEMMA5.1},
we have
$$
\|\delta_{\ell\hat{n}}\psi\|_{L^3_\lambda}
\leq
C \ell^{1-s}
\left(
\|\theta\|_{L^3_\lambda}
+
\|\theta\|^{\mathrm{low}}_{\dot B^{-s}_{3,\infty}}
\right).$$
Therefore,
$$
|S_H(\ell)|
\leq
C\ell^{1-s}
\mathbb E
\left[
\|\theta\|^2_{L^3_\lambda}
\left(
\|\theta\|_{L^3_\lambda}
+
\|\theta\|^{\mathrm{low}}_
{\dot B^{-s}_{3,\infty}}
\right)
\right].
$$
By the Besov assumption, there exists $C_s<\infty$ such that
$$
|S_H(\ell)|
\leq
C_s \ell^{1-s}
$$
uniformly in $\nu,\alpha$. Hence
$$
\sup_{\ell\in[\ell_I,\lambda)}
\left|
\frac{S_H(\ell)}{\ell}
\right|
\leq
C_s
\sup_{\ell\in[\ell_I,\lambda)}
\ell^{-s}
\leq
C_s \ell_I^{-s},\ 0<s<1.
$$
Taking $\limsup_{\nu,\alpha\to0}$ 
and then letting $\ell_I\to\infty$, we obtain
$$
\lim_{\ell_I\to\infty}
\limsup_{\nu,\alpha\to0}
\sup_{\ell\in[\ell_I,\lambda)}
\left|
\frac{S_H(\ell)}{\ell}
\right|
=0.
$$
This completes the proof.
\end{proof}

\appendix
\section{Proof of Proposition \ref{Prop2.4}}\label{app:main}

Here we give the proof of Proposition \ref{Prop2.4}.
\begin{proof}
Let
$
X(t):=\norm{\theta(t)}_{L^2_\lambda}^2.
$
We note that all 
applications of Itô's 
formula below can be
justified by 
a standard smooth
truncation argument
and a subsequent passage to the limit.

Applying Itô’s formula to $X(t)$ and using the cancellation
$$
\ip{u\cdot\nabla\theta}{\theta}_{\lambda}=0,
$$
we obtain
$$
dX(t)+2\left\langle \theta,\mathfrak{B}_{\nu,\alpha}\theta\right\rangle _{\lambda}\dd t
=
2\ep\,\dd t
+
2\sum_{n\in\mathbb N}\left\langle \theta,g_n^\lambda\right\rangle_{\lambda}\dd W_n(t),
$$
where we denote that
$\mathfrak{B}_{\nu,\alpha}:=\nu\La^2+\alpha\La^{-2\gamma},$
$\ip{f}{g}_{\lambda}:=\fint_{\T_\lambda^2} f(x)g(x)\,dx
$ is the normalized inner product,
and recall that
$
\ep=\frac12\sum_{n\in\mathbb N}\norm{g_n^\lambda}_{L^2_\lambda}^2
$.

\noindent By stationarity, taking 
expectation gives the stationary SPE balance
$$
\nu\E\norm{\La\theta}_{L^2_\lambda}^2
+
\alpha\E\norm{\La^{-\gamma}\theta}_{L^2_\lambda}^2
=
\ep.
$$
In particular,
$$
\E\norm{\La\theta}_{L^2_\lambda}^2<\infty.
$$

We next prove a fourth moment bound in $L^2_\lambda$. Since the functions have zero
spatial mean, the operator $\mathfrak{B}_{\nu,\alpha}$ is coercive on the zero-mean subspace. More
precisely, for fixed $\nu,\alpha,\lambda$,
$$
\left\langle f,\mathfrak{B}_{\nu,\alpha}f \right\rangle _{\lambda}
\ge
c_{\nu,\alpha,\lambda}\norm{f}_{L^2_\lambda}^2,
\qquad
c_{\nu,\alpha,\lambda}:=
\inf_{k\in\mathbb Z_\lambda^2\setminus\{0\}}
\br{\nu |k|^2+\alpha |k|^{-2\gamma}}>0.
$$
Applying Itô’s formula to $X(t)^2$, we get
$$
dX(t)^2
=
2X(t)\,dX(t)+d\langle X,X\rangle_t.
$$
Moreover,
$$
d\langle X,X\rangle_t
=
4\sum_{n\in\mathbb N}\left\langle \theta,g_n^\lambda\right\rangle_{\lambda}^2\dd t
\le
4X(t)\sum_{n\in\mathbb N}\norm{g_n^\lambda}_{L^2_\lambda}^2\dd t
=
8\ep X(t)\dd t.
$$
Therefore, by stationarity,
$$
4c_{\nu,\alpha,\lambda}\E X^2
\le
12\ep \E X.
$$
Since the coercivity and the stationary energy balance also imply
$
c_{\nu,\alpha,\lambda}\E X\le \ep,
$
we conclude that
$$
\E\norm{\theta}_{L^2_\lambda}^4=\E X^2<\infty.
$$

By the two-dimensional Gagliardo--Nirenberg inequality on the fixed torus
$\T_\lambda^2$,
$$
\norm{\theta}_{L^3_\lambda}^3
\lesssim_\lambda
\norm{\theta}_{L^2_\lambda}^2
\norm{\La\theta}_{L^2_\lambda}.
$$
Hence, by the Cauchy--Schwarz inequality,
$$
\E\norm{\theta}_{L^3_\lambda}^3
\lesssim_\lambda
\br{\E\norm{\theta}_{L^2_\lambda}^4}^{1/2}
\br{\E\norm{\La\theta}_{L^2_\lambda}^2}^{1/2}
<\infty.
$$

It remains to check the finiteness of the two 
structure functions appearing in Theorem \ref{Thm}.
Since $u=R^\perp\theta$ and
the Riesz transforms are bounded on $L^3$, we have
$$
\begin{aligned}
\fint_{\T_\lambda^2}
|\delta_h\theta(x)|^2|\delta_hu(x)|\dx
\leq
\norm{\delta_h\theta}_{L^3_\lambda}^2
\norm{\delta_hu}_{L^3_\lambda}  
\lesssim
\norm{\theta}_{L^3_\lambda}^2\norm{u}_{L^3_\lambda} 
\lesssim
\norm{\theta}_{L^3_\lambda}^3.
\end{aligned}
$$
Taking expectations gives the finiteness of $S_E(\ell)$.

Similarly, since $\psi=\La^{-1}\theta$ and $\La^{-1}$ is bounded from the zero-mean
$L^3$ space into itself on the fixed torus $\T_\lambda^2$,
$$
\norm{\psi}_{L^3_\lambda}\lesssim_\lambda \norm{\theta}_{L^3_\lambda}.
$$
Therefore,
$$
\begin{aligned}
&\fint_
{\mathbb{T}_\lambda^2} |\delta_{h}\theta|
 |\delta_{h}\psi|
|\delta_{h}u|\,dx   
\lesssim
\norm{\theta}_{L^3_\lambda}
\norm{\psi}_{L^3_\lambda}
\norm{u}_{L^3_\lambda} 
\lesssim_\lambda
\norm{\theta}_{L^3_\lambda}^3.
\end{aligned}
$$
Taking expectations gives the finiteness of $S_H(\ell)$. 
This completes the proof.
\end{proof}

\medskip

\section*{Declarations}

\noindent\textbf{Data Availability} 
No data has been produced in the original research reported in this manuscript.

\medskip

\noindent\textbf{Conflict of interest} 
The authors have no conflict of interest 
to declare that are relevant to the content of this
article.

\medskip

\bibliographystyle{plain}
\bibliography{ref}

\end{document}